\documentclass{article}%
\usepackage{amsfonts}
\usepackage{amsmath}
\usepackage{amssymb}
\usepackage{graphicx}%
\providecommand{\U}[1]{\protect\rule{.1in}{.1in}}
\newtheorem{theorem}{Theorem}

\newtheorem{corollary}[theorem]{Corollary}

\newtheorem{proposition}[theorem]{Proposition}

\newenvironment{proof}[1][Proof]{\noindent\textbf{#1.} }{\ \rule{0.5em}{0.5em}}
\begin{document}

\title{Higher order generalized geometric polynomials}
\author{Levent Karg\i n* and Bayram \c{C}ekim**
\and * Alanya Alaaddin Keykubat University\\Akseki Vocational School TR-07630 Antalya Turkey\\Gazi University, Faculty of Science,\\Department of Mathematics, Teknikokullar TR-06500,\\Ankara, Turkey.\\leventkargin48@gmail.com.tr and **bayramcekim@gazi.edu.tr}
\maketitle

\begin{abstract}
According to generalized Mellin derivative \cite[Eq. (2.5)]{Kargin}, we
introduce a new family of polynomials called higher order generalized
geometric polynomials. We obtain some properties of them.We discuss their
connections to degenerate Bernoulli and Euler polynomials. Furthermore, we
find new formulas for the Carlitz's \cite[Eq. (5.4)]{Carlitz} and Howard's
\cite[Eq. (4.3)]{Howard2} finite sums. Finally, we evaluate several series in
closed forms, one of which has the coefficients include values of the Riemann
zeta function. Moreover, we calculate some integrals in terms of generalized
geometric polynomials.



\end{abstract}

\section{Introduction}

The operator $\left(  x\frac{d}{dx}\right)  ^{n}$, called Mellin derivative
\cite{B2}, has a long mathematical history. As far back as 1740, Euler used
the operator as a tool work in his work \cite{Euler}. The Mellin derivative
and its generalizations are used to obtain a new class of polynomials \cite{B,
Dil1, Dil2, Kargin, MezoandRamirez}, to evaluate some power series in closed
forms \cite{B, BandDil, Dil1, Dil2, Kargin, Knopf, MezoandRamirez} and to
calculate some integrals \cite{B2, BandDil}. One of the generalizations of
Mellin derivative is
\begin{equation}
\left(  \beta x^{1-\alpha/\beta}D\right)  ^{n}\left[  x^{r/\beta}f\left(
x\right)  \right]  =x^{\left(  r-n\alpha\right)  /\beta}\sum_{k=0}^{n}S\left(
n,k;\alpha,\beta,r\right)  \beta^{k}x^{k}f^{\left(  k\right)  }\left(
x\right)  ,\label{1}%
\end{equation}
where $f$ is any $n$-times differentiable function and $S\left(
n,k;\alpha,\beta,r\right)  $ are generalized Stirling number pair with three
free parameters (see Section 2). Stirling numbers and their generalizations
have many interesting combinatorial interpretations. Besides, these numbers
are connected with some well known special polynomials and numbers
\cite{Cenkci1, Cenkci2, Dagli1, Howard1, Merca1, Merca2, Merca3, Mezo2,
Mihioubi2, Young, Young2}. For example, the following interesting formulas for
Bernoulli numbers $B_{n}$ and Euler polynomials $E_{n}\left(  x\right)  $
appear in \cite{Graham}: For all $n\geq0$;%
\begin{equation}
B_{n}=\sum_{k=0}^{n}\left(  -1\right)  ^{k}\frac{k!}{k+1}%
\genfrac{\{}{\}}{0pt}{}{n}{k}%
,\text{ \ }E_{n}\left(  0\right)  =\sum_{k=0}^{n}\left(  -1\right)  ^{k}%
\frac{k!}{2^{k}}%
\genfrac{\{}{\}}{0pt}{}{n}{k}%
.\label{14}%
\end{equation}

From all these motivations, by using the generalized of Mellin derivative in
(\ref{1}), we introduce a new family of polynomials $w_{n}^{\left(
s+1\right)  }\left(  x;\alpha,\beta,r\right)  $, called higher order
generalized geometric polynomials, evaluate some power series in closed forms
and calculate some integrals. After that, in view of the properties of higher
order generalized geometric polynomials, we derive new explicit formulas for
degenerate Bernoulli polynomials, Bernoulli polynomials and degenerate Euler
polynomials. As a consequences of these formulas we evaluate Carlitz's
\cite[Eq. (5.4)]{Carlitz} and Howard's \cite[Eq. (4.3)]{Howard2} sums.

The summary by sections is as follows: Section 2 is the preliminary section
where we give definitions and known results needed. In Section 3, we define
higher order generalized geometric polynomials and obtain some properties such
as recurrence relation and generating function of $w_{n}^{\left(  s+1\right)
}\left(  x;\alpha,\beta,r\right)  .$ Moreover, we derive new explicit formulas
for degenerate Bernoulli polynomials, Bernoulli polynomials and degenerate
Euler polynomials. By the help of these formulas we find new formulas for
Carlit's and Howard's sums. In final section, we obtain some integral
representation of $w_{n}^{\left(  s+1\right)  }\left(  x;\alpha,\beta
,r\right)  $. Besides, we evaluate several power series in closed forms. For
example, in the special cases, the following series, where the coefficients
include values of the Riemann zeta function, is evaluated in closed form (for
any integer $n\geq0$)%
\[
\sum_{k=2}^{\infty}\frac{\zeta\left(  k\right)  k^{n}}{2^{k}}=\log2+\sum
_{k=1}^{n}%
\genfrac{\{}{\}}{0pt}{}{n+1}{k+1}%
k!\left(  1-2^{-k-1}\right)  \zeta\left(  k+1\right)  .
\]

Throughout this paper we assume that $\alpha,\beta$ and $r$ are real or
complex numbers.

\section{Preliminaries}

The generalized Stirling numbers of the first kind $S_{1}\left(
n,k;\alpha,\beta,r\right)  $ and of the second kind $S_{2}\left(
n,k;\alpha,\beta,r\right)  $ for non-negative integer $m$ and real or complex
parameters $\alpha,$ $\beta$ and $r,$ with $\left(  \alpha,\beta,r\right)
\neq\left(  0,0,0\right)  $ are defined by means of the generating function
\cite{HSU}%
\begin{align*}
\left(  \frac{\left(  1+\beta t\right)  ^{\alpha/\beta}-1}{\alpha}\right)
^{k}\left(  1+\beta t\right)  ^{r/\beta}  &  =k!\sum_{n=0}^{\infty}%
S_{1}\left(  n,k;\alpha,\beta,r\right)  \frac{t^{n}}{n!},\\
\left(  \frac{\left(  1+\alpha t\right)  ^{\beta/\alpha}-1}{\beta}\right)
^{k}\left(  1+\alpha t\right)  ^{-r/\alpha}  &  =k!\sum_{n=0}^{\infty}%
S_{2}\left(  n,k;\alpha,\beta,r\right)  \frac{t^{n}}{n!},
\end{align*}
with the convention $S_{1}\left(  n,k;\alpha,\beta,r\right)  =S_{2}\left(
n,k;\alpha,\beta,r\right)  =0$ when $k>n.$

As Hsu and Shiue pointed out, the definitions or generating functions
generalize various Stirling-type numbers studied previously, such as:

\begin{description}
\item[i.] $\left\{  S_{1}\left(  n,k;0,1,0\right)  ,\text{ }S_{2}\left(
n,k;0,1,0\right)  \right\}  =\left\{  s\left(  n,k\right)  ,\text{ }S\left(
n,k\right)  \right\}  =\left\{  \left(  -1\right)  ^{n-k}%
\genfrac{[}{]}{0pt}{}{n}{k}%
,\text{ }%
\genfrac{\{}{\}}{0pt}{}{n}{k}%
\right\}  $

\item are the Stirling numbers of both kinds \cite[Chapter 6]{Graham}.

\item[ii.] $\left\{  S_{1}\left(  n,k;\alpha,1,-r\right)  ,\text{ }%
S_{2}\left(  n,k;\alpha,1,-r\right)  \right\}  =\left\{  \left(  -1\right)
^{n-k}S_{1}\left(  n,k,r+\alpha\mid\alpha\right)  ,\text{ }S_{2}\left(
n,k,r\mid\alpha\right)  \right\}  $

\item are the Howard degenerate weighted Stirling numbers of both kinds
\cite{Howard}.

\item[iii.] $\left\{  S_{1}\left(  n,k;\alpha,1,0\right)  ,\text{ }%
S_{2}\left(  n,k;\alpha,1,0\right)  \right\}  =\left\{  \left(  -1\right)
^{n-k}S_{1}\left(  n,k\mid\alpha\right)  ,\text{ }S_{2}\left(  n,k\mid
\alpha\right)  \right\}  $

\item are the Carlitz's degenerate Stirling numbers of both kinds
\cite{Carlitz}.

\item[iv.] $\left\{  S_{1}\left(  n,k;0,-1,r\right)  ,\text{ }S_{2}\left(
n,k;0,1,r\right)  \right\}  =\left\{  \left(  -1\right)  ^{n-k}%
\genfrac{[}{]}{0pt}{}{n+r}{k+r}%
_{r},\text{ }%
\genfrac{\{}{\}}{0pt}{}{n+r}{k+r}%
_{r}\right\}  $

\item are the $r$-Stirling numbers of both kinds \cite{Broder}.

\item[v.] $\left\{  S_{1}\left(  n,k;0,\beta,-1\right)  ,\text{ }S_{2}\left(
n,k;0,\beta,1\right)  \right\}  =\left\{  w_{\beta}\left(  n,k\right)  ,\text{
}W_{\beta}\left(  n,k\right)  \right\}  $

\item are the Whitney numbers of the first and second kind \cite{Benoumhani2}.

\item[vi.] $\left\{  S_{1}\left(  n,k;0,\beta,-r\right)  ,\text{ }S_{2}\left(
n,k;0,\beta,r\right)  \right\}  =\left\{  w_{\beta,r}\left(  n,k\right)
,\text{ }W_{\beta,r}\left(  n,k\right)  \right\}  $

\item are the $r$-Whitney numbers of the first and second kind \cite{Mezo2},
\end{description}

and so on.

According to the generalization of Stirling numbers, Hsu and Shiue \cite{HSU}
defined generalized exponential polynomials $S_{n}\left(  x\right)  $ as
follows%
\begin{equation}
S_{n}\left(  x\right)  =\sum_{k=0}^{n}S\left(  n,k;\alpha,\beta,r\right)
x^{k}. \label{6}%
\end{equation}
Later, Kargin and Corcino \cite{Kargin} gave an equivalent definition for
$S_{n}\left(  x\right)  $ as%
\begin{equation}
S_{n}\left(  x\right)  =\left[  \left(  x^{n\alpha-r}e^{-x}\right)  \right]
^{1/\beta}\left(  \beta x^{1-\alpha/\beta}D\right)  ^{n}\left[  \left(
x^{r}e^{x}\right)  ^{1/\beta}\right]  , \label{15}%
\end{equation}
and using the series%
\[
e^{x/\beta}=\sum_{k=0}^{\infty}\frac{x^{k}}{\beta^{k}k!},
\]
in (\ref{15}), they obtained the general Dobinksi-type formula \cite{HSU}%
\begin{equation}
e^{x/\beta}S_{n}\left(  x\right)  =\sum_{k=0}^{\infty}\frac{\left(
k\beta+r\mid\alpha\right)  _{n}x^{k}}{\beta^{k}k!}. \label{16}%
\end{equation}
Here, $\left(  z\mid\alpha\right)  _{n}$ is called the generalized factorial
of $z$ with increment $\alpha$, defined by $\left(  z\mid\alpha\right)
_{n}=z\left(  z-\alpha\right)  \cdots\left(  z-n\alpha+\alpha\right)  $ for
$n=1,2,\ldots,$ and $\left(  z\mid\alpha\right)  _{0}=1.$ In particular, we
have $\left(  z\mid1\right)  _{n}=\left(  z\right)  _{n}.$

Some other properties of $S_{n}\left(  x\right)  $ can be found in
\cite{Corcino4, Corcino5, HSU, Aimin}.

The generalized geometric polynomials $w_{n}\left(  x;\alpha,\beta,r\right)  $
are defined by means of the generalized Mellin derivative as%
\[
\left(  \beta x^{1-\alpha/\beta}D\right)  ^{n}\left[  \frac{x^{r/\beta}}%
{1-x}\right]  =\frac{x^{\left(  r-n\alpha\right)  /\beta}}{1-x}w_{n}\left(
\frac{x}{1-x};\alpha,\beta,r\right)  .
\]
These polynomials have the explicit formula
\begin{equation}
w_{n}\left(  x;\alpha,\beta,r\right)  =\sum_{k=0}^{n}S\left(  n,k;\alpha
,\beta,r\right)  \beta^{k}k!x^{k}, \label{11}%
\end{equation}
and the generating function
\begin{equation}
\sum_{n=0}^{\infty}w_{n}\left(  x;\alpha,\beta,r\right)  \frac{t^{n}}%
{n!}=\frac{\left(  1+\alpha t\right)  ^{r/\alpha}}{1-x\left(  \left(  1+\alpha
t\right)  ^{\beta/\alpha}-1\right)  },\text{ \ }\alpha\beta\neq0. \label{19}%
\end{equation}
See \cite{Kargin} for details.

The higher order degenerate Euler polynomials are defined by means of the
generating function \cite{Carlitz}%
\begin{equation}
\sum_{n=0}^{\infty}\mathcal{E}_{n}^{\left(  s\right)  }\left(  \alpha
;x\right)  \frac{t^{n}}{n!}=\left(  \frac{2}{\left(  1+\alpha t\right)
^{1/\alpha}+1}\right)  ^{s}\left(  1+\alpha t\right)  ^{x/\alpha}. \label{9}%
\end{equation}

From (\ref{9}), we have%
\[
\underset{\alpha\rightarrow0}{\lim}\mathcal{E}_{n}^{\left(  s\right)  }\left(
\alpha;x\right)  =E_{n}^{\left(  s\right)  }\left(  x\right)  ,
\]
where $E_{n}^{\left(  s\right)  }\left(  x\right)  $ are the higher order
Euler polynomials which are defined by the generating function%
\begin{equation}
\sum_{n=0}^{\infty}E_{n}^{\left(  s\right)  }\left(  x\right)  \frac{t^{n}%
}{n!}=\left(  \frac{2}{e^{t}+1}\right)  ^{s}e^{xt}. \label{28}%
\end{equation}

In special cases,%
\[
\mathcal{E}_{n}^{\left(  1\right)  }\left(  \alpha;x\right)  =\mathcal{E}%
_{n}\left(  \alpha;x\right)  ,\text{ }E_{n}^{\left(  1\right)  }\left(
x\right)  =E_{n}\left(  x\right)  ,
\]
where $\mathcal{E}_{n}\left(  \alpha;x\right)  $ and $E_{n}\left(  x\right)  $
are the degenerate Euler and Euler polynomials, respectively.

The degenerate Bernoulli polynomials of the second kind are defined by means
of the generating function \cite{Kim}%
\begin{equation}
\sum_{n=0}^{\infty}B_{n}\left(  x\mid\alpha\right)  \frac{t^{n}}{n!}%
=\frac{\frac{1}{\alpha}\log\left(  1+\alpha t\right)  }{\left(  1+\alpha
t\right)  ^{1/\alpha}-1}\left(  1+\alpha t\right)  ^{x/\alpha}. \label{34}%
\end{equation}

Indeed, we get%
\[
\underset{\alpha\rightarrow0}{\lim}B_{n}\left(  x\mid\alpha\right)
=B_{n}\left(  x\right)  ,
\]
where $B_{n}\left(  x\right)  $ are the Bernoulli polynomials which are
defined by the generating function%
\[
\sum_{n=0}^{\infty}B_{n}\left(  x\right)  \frac{t^{n}}{n!}=\frac{t}{e^{t}%
-1}e^{xt},
\]
with $B_{n}\left(  0\right)  =B_{n}$ are $n$-the Bernoulli numbers.

Finally, we want to mention Carlitz's degenerate Bernoulli polynomials defined
by means of the generating function \cite{Carlitz}%
\[
\sum_{n=0}^{\infty}\beta_{n}\left(  \alpha,x\right)  \frac{t^{n}}{n!}=\frac
{t}{\left(  1+\alpha t\right)  ^{1/\alpha}-1}\left(  1+\alpha t\right)
^{x/\alpha},
\]
with the relation
\[
\underset{\alpha\rightarrow0}{\lim}\beta_{n}\left(  \alpha,x\right)
=B_{n}\left(  x\right)
\]
and with $\beta_{n}\left(  \alpha,0\right)  =\beta_{n}\left(  \alpha\right)  $
are $n$-the degenerate Bernoulli numbers.

\section{Higher order generalized geometric polynomials}

In this section, the definition of higher order generalized geometric
polynomials and some properties are given. Besides, new explicit formulas for
degenerate Bernoulli and Euler polynomials are derived.\textbf{ }Some special
cases of these results are studied.

For every $s\geq0,$ taking $f\left(  x\right)  =1/\left(  1-x\right)  ^{s+1}$
in (\ref{1}) and using
\[
\frac{\partial^{k}}{\partial x^{k}}f\left(  x\right)  =\frac{\left(
s+1\right)  \left(  s+2\right)  \ldots\left(  s+k\right)  }{\left(
1-x\right)  ^{s+k+1}},
\]
we have
\[
\left(  \beta x^{1-\alpha/\beta}D\right)  ^{n}\left[  \frac{x^{r/\beta}%
}{\left(  1-x\right)  ^{s+1}}\right]  =\frac{x^{\left(  r-n\alpha\right)
/\beta}}{\left(  1-x\right)  ^{s+1}}\sum_{k=0}^{n}S\left(  n,k;\alpha
,\beta,r\right)  \binom{s+k}{k}k!\beta^{k}\left(  \frac{x}{1-x}\right)  ^{k}.
\]

If we define the polynomials $w_{n}^{\left(  s+1\right)  }\left(
x;\alpha,\beta,r\right)  $ by
\begin{equation}
w_{n}^{\left(  s+1\right)  }\left(  x;\alpha,\beta,r\right)  =\sum_{k=0}%
^{n}S\left(  n,k;\alpha,\beta,r\right)  \binom{s+k}{k}k!\beta^{k}x^{k},
\label{3}%
\end{equation}
we find
\begin{equation}
\left(  \beta x^{1-\alpha/\beta}D\right)  ^{n}\left[  \frac{x^{r/\beta}%
}{\left(  1-x\right)  ^{s+1}}\right]  =\frac{x^{\left(  r-n\alpha\right)
/\beta}}{\left(  1-x\right)  ^{s+1}}w_{n}^{\left(  s+1\right)  }\left(
\frac{x}{1-x};\alpha,\beta,r\right)  . \label{4}%
\end{equation}

Note that if $\left(  s,\alpha,\beta,r\right)  =\left(  0,\alpha
,\beta,r\right)  ,$ we obtain generalized geometric polynomials \cite{Kargin},
if $\left(  s,\alpha,\beta,r\right)  =\left(  s,0,1,0\right)  ,$ we have
general geometric polynomials \cite{B}, if $\left(  s,\alpha,\beta,r\right)
=\left(  0,0,1,0\right)  ,$ we obtain geometric polynomials \cite{B}, if
$\left(  s,\alpha,\beta,r\right)  =\left(  0,0,\beta,1\right)  ,$ we have
Tanny-Dowling polynomials \cite{Benoumhani2}.

For $x=1$ and $s=0$ in (\ref{3}) gives the numbers%
\[
w_{n}^{\left(  1\right)  }\left(  1;\alpha,\beta,r\right)  =B_{n}\left(
\alpha,\beta,r\right)  =\sum_{k=0}^{n}S\left(  n,k;\alpha,\beta,r\right)
\beta^{k}k!x^{k},
\]
which was defined by Corcino et al. \cite{Corcino6}. The combinatorial
interpretation and some other properties can be found in \cite{Corcino6,
Corcino4}. Moreover, $w_{n}^{\left(  s+1\right)  }\left(  x;\alpha
,\beta,r\right)  $ reduce to barred preferential arrangement numbers $r_{n,s}$
defined by \cite{Albach, B}%
\[
w_{n}^{\left(  s+1\right)  }\left(  1;0,1,0\right)  =r_{n,s}=\sum_{k=0}^{n}%
\genfrac{\{}{\}}{0pt}{}{n}{k}%
\binom{s+k}{k}k!,
\]
which have interesting combinatorial meaning. Therefore, we may call
generalized barred preferential arrangement number pair with three free
parameters for
\[
B_{n}^{\left(  s+1\right)  }\left(  \alpha,\beta,r\right)  =\sum_{k=0}%
^{n}S\left(  n,k;\alpha,\beta,r\right)  \binom{s+k}{k}k!\beta^{k}x^{k}.
\]
The combinatorial interpretation of these numbers may also be studied.

On the other hand, we can write $w_{n}^{\left(  s+1\right)  }\left(
x;\alpha,\beta,r\right)  $ in the form%
\[
w_{n}^{\left(  s+1\right)  }\left(  x;\alpha,\beta,r\right)  =\sum_{k=0}%
^{n}S\left(  n,k;\alpha,\beta,r\right)  \left\langle s+1\right\rangle
_{k}\beta^{k}x^{k},
\]
since
\[
\frac{\left\langle x\right\rangle _{n}}{n!}=\binom{x+n-1}{n},
\]
where $\left\langle x\right\rangle _{n}$ is the rising factorial defined by
$\left\langle x\right\rangle _{n}=x\left(  x+1\right)  \cdots\left(
x+n-1\right)  ,$ for $n=1,2,\ldots,$ with $\left\langle x\right\rangle
_{0}=1.$\textbf{ }Furthermore, using the relation
\begin{equation}
\left\langle -x\right\rangle _{n}=\left(  -1\right)  ^{n}\left(  x\right)
_{n}, \label{36}%
\end{equation}
if we define $w_{n}^{\left(  -s\right)  }\left(  x;\alpha,\beta,r\right)  $
for every $s>0$ as
\[
w_{n}^{\left(  -s\right)  }\left(  x;\alpha,\beta,r\right)  =\sum_{k=0}%
^{n}S\left(  n,k;\alpha,\beta,r\right)  \left(  s\right)  _{k}\left(
-\beta\right)  ^{k}x^{k},
\]
we find
\begin{equation}
\left(  \beta x^{1-\alpha/\beta}D\right)  ^{n}\left[  x^{r/\beta}\left(
1-x\right)  ^{s}\right]  =x^{\left(  r-n\alpha\right)  /\beta}\left(
1-x\right)  ^{s}w_{n}^{\left(  -s\right)  }\left(  \frac{x}{1-x};\alpha
,\beta,r\right)  \label{38}%
\end{equation}
by taking $f\left(  x\right)  =\left(  1-x\right)  ^{s}$ in (\ref{1}).

We turn over $w_{n}^{\left(  -s\right)  }\left(  x;\alpha,\beta,r\right)  $ in
Section 4.

Now, we want to deal with the properties of $w_{n}^{\left(  s\right)  }\left(
x;\alpha,\beta,r\right)  .$ The exponential and higher order generalized
geometric polynomials are connected by the relation (for any non-negative
integer $n$ and every $s>0$)%
\begin{equation}
w_{n}^{\left(  s\right)  }\left(  x;\alpha,\beta,r\right)  =\frac{1}%
{\Gamma\left(  s\right)  }%
{\displaystyle\int\limits_{0}^{\infty}}
z^{s-1}S_{n}\left(  x\beta z\right)  e^{-z}dz, \label{7}%
\end{equation}
which is verified immediately by using (\ref{6}). So, we can derive the
properties of $w_{n}^{\left(  s\right)  }\left(  x;\alpha,\beta,r\right)  $
from those of $S_{n}\left(  x\right)  $. For example, if we use the extension
of Spivey's Bell number formula to $S_{n}\left(  x\right)  $ \cite{Kargin,
Aimin}%
\[
S_{n+m}\left(  x\right)  =\sum_{k=0}^{n}\sum_{j=0}^{m}\binom{n}{k}S\left(
m,j;\alpha,\beta,r\right)  \left(  j\beta-m\alpha\mid\alpha\right)
_{n-k}S_{k}\left(  x\right)  x^{j}%
\]
in (\ref{7}) and evaluate the the integral, we derive a recurrence relation
as
\begin{align*}
&  w_{n+m}^{\left(  s\right)  }\left(  x;\alpha,\beta,r\right) \\
&  \quad=\sum_{k=0}^{n}\sum_{j=0}^{m}\binom{n}{k}S\left(  m,j;\alpha
,\beta,r\right)  \left(  j\beta-m\alpha\mid\alpha\right)  _{n-k}\left\langle
s+1\right\rangle _{j}\beta^{j}w_{k}^{\left(  s+j\right)  }\left(
x;\alpha,\beta,r\right)  x^{j}.
\end{align*}

As another application of (\ref{7}), we give the following theorem.

\begin{theorem}
The exponential generating function for $w_{n}^{\left(  s\right)  }\left(
x;\alpha,\beta,r\right)  $ is
\begin{equation}
\sum_{n=0}^{\infty}w_{n}^{\left(  s\right)  }\left(  x;\alpha,\beta,r\right)
\frac{t^{n}}{n!}=\left(  \frac{1}{1-x\left(  \left(  1+\alpha t\right)
^{\beta/\alpha}-1\right)  }\right)  ^{s}\left(  1+\alpha t\right)  ^{r/\alpha
}, \label{8}%
\end{equation}
where $\alpha\beta\neq0.$
\end{theorem}

\begin{proof}
From \cite[Eq. (12)]{HSU}, let us write the generating function for
$S_{n}\left(  x\right)  $ in the form
\[
\sum_{n=0}^{\infty}S_{n}\left(  x\beta z\right)  \frac{t^{n}}{n!}=\left(
1+\alpha t\right)  ^{r/\alpha}\exp\left[  xz\left(  \left(  1+\alpha t\right)
^{\beta/\alpha}-1\right)  \right]  .
\]
Then multiply both sides by $z^{s-1}e^{-z}$ and integrate for $z$ from zero to
infinity. In view of (\ref{7}) this gives
\[
\sum_{n=0}^{\infty}w_{n}^{\left(  s\right)  }\left(  x;\alpha,\beta,r\right)
\frac{t^{n}}{n!}=\frac{\left(  1+\alpha t\right)  ^{r/\alpha}}{\Gamma\left(
s\right)  }%
{\displaystyle\int\limits_{0}^{\infty}}
z^{s-1}e^{-z\left(  1-x\left(  \left(  1+\alpha t\right)  ^{\beta/\alpha
}-1\right)  \right)  }dz.
\]
Calculating the integral on the right-hand side completes the proof.
\end{proof}

Setting $x=-1$ in (\ref{8}), we have%
\[
w_{n}^{\left(  s\right)  }\left(  -1;\alpha,\beta,r\right)  =w_{n}\left(
-1;\alpha,\beta,r\right)  =\left(  r-\beta s\mid\alpha\right)  _{n}.
\]

Some other properties of $w_{n}^{\left(  s\right)  }\left(  x;\alpha
,\beta,r\right)  $ can be derived from (\ref{8}).

Now, we attend to the connection of $w_{n}^{\left(  s\right)  }\left(
x;\alpha,\beta,r\right)  $ with some degenerate special polynomials.

If we take $x=-1/2$ in (\ref{8}) for $\beta=1$ and compare with (\ref{9}), we
have%
\[
w_{n}^{\left(  s\right)  }\left(  \frac{-1}{2};\alpha,1,r\right)
=\mathcal{E}_{n}^{\left(  s\right)  }\left(  \alpha;r\right)  .
\]
Thus, using (\ref{3}) yields an explicit formula for higher order degenerate
Euler polynomials in the following corollary.

\begin{corollary}
For every $s\geq0,$ we have%
\begin{equation}
\mathcal{E}_{n}^{\left(  s+1\right)  }\left(  \alpha;r\right)  =\sum_{k=0}%
^{n}S_{2}\left(  n,k,r\mid\alpha\right)  \frac{\left(  -1\right)
^{k}\left\langle s+1\right\rangle _{k}}{2^{k}}. \label{10}%
\end{equation}
When $s=0$, this becomes%
\begin{equation}
\mathcal{E}_{n}\left(  \alpha;r\right)  =\sum_{k=0}^{n}S_{2}\left(
n,k,r\mid\alpha\right)  \frac{\left(  -1\right)  ^{k}k!}{2^{k}}. \label{27}%
\end{equation}

\end{corollary}

To the writer's knowledge, this is a new result. (\ref{10}) and (\ref{27}) are
new results.

From (\ref{8}) and (\ref{28}), one can obtain%
\[
\underset{\alpha\rightarrow0}{\lim}w_{n}^{\left(  s\right)  }\left(  \frac
{-1}{2};\alpha,\beta,r\right)  =E_{n}^{\left(  s\right)  }\left(  \frac
{r}{\beta}\right)  \beta^{n}.
\]
Thus, we have
\[
E_{n}^{\left(  s\right)  }\left(  \frac{r}{\beta}\right)  =\sum_{k=0}%
^{n}W_{\beta,r}\left(  n,k\right)  \frac{\left(  -1\right)  ^{k}\left\langle
s\right\rangle _{k}}{\beta^{n-k}2^{k}},
\]
which was given in \cite{Mihioubi2} with a different proof. Moreover, for
$\beta=1$ and $r=0$ the above equation reduce to the right part of (\ref{14}).

Secondly, if we integrate both sides of (\ref{19}) for $x$ from $-1$ to $0,$
we have%
\begin{equation}
\sum_{n=0}^{\infty}\frac{t^{n}}{n!}%
{\displaystyle\int\limits_{-1}^{0}}
w_{n}\left(  x;\alpha,\beta,r\right)  dx=\frac{\frac{\beta}{\alpha}\log\left(
1+\alpha t\right)  }{\left(  1+\alpha t\right)  ^{\beta/\alpha}-1}\left(
1+\alpha t\right)  ^{r/\alpha}. \label{20}%
\end{equation}
In view of (\ref{34}), for $\beta=1,$ the above equation becomes%
\[
\sum_{n=0}^{\infty}\frac{t^{n}}{n!}%
{\displaystyle\int\limits_{-1}^{0}}
w_{n}\left(  x;\alpha,1,r\right)  dx=\sum_{n=0}^{\infty}B_{n}\left(
r\mid\alpha\right)  \frac{t^{n}}{n!}.
\]
Comparing the coefficients of $\frac{t^{n}}{n!}$ gives
\[%
{\displaystyle\int\limits_{-1}^{0}}
w_{n}\left(  x;\alpha,1,r\right)  dx=B_{n}\left(  r\mid\alpha\right)  .
\]
Finally, using (\ref{11}) yields the following theorem.

\begin{theorem}
The following equation holds for degenerate Bernoulli polynomials of the
second kind%
\[
B_{n}\left(  r\mid\alpha\right)  =\sum_{k=0}^{n}S_{2}\left(  n,k,r\mid
\alpha\right)  \frac{\left(  -1\right)  ^{k}k!}{k+1}.
\]

\end{theorem}

We note that for $\alpha\rightarrow0,$ (\ref{20}) can be written as
\begin{align*}
\underset{\alpha\rightarrow0}{\lim}\sum_{n=0}^{\infty}\frac{t^{n}}{n!}%
{\displaystyle\int\limits_{-1}^{0}}
w_{n}\left(  x;\alpha,\beta,r\right)  dx  &  =\frac{\beta t}{e^{\beta t}%
-1}e^{rt}\\
&  =\sum_{n=0}^{\infty}B_{n}\left(  \frac{r}{\beta}\right)  \beta^{n}%
\frac{t^{n}}{n!}.
\end{align*}
Comparing the coefficients of $\frac{t^{n}}{n!}$ in the above equation, we
obtain%
\[
\underset{\alpha\rightarrow0}{\lim}%
{\displaystyle\int\limits_{-1}^{0}}
w_{n}\left(  x;\alpha,\beta,r\right)  dx=B_{n}\left(  \frac{r}{\beta}\right)
\beta^{n}.
\]
Since $S_{2}\left(  n,k;0,\beta,r\right)  =W_{\beta,r}\left(  n,k\right)  $ we
find
\[
B_{n}\left(  \frac{r}{\beta}\right)  =\sum_{k=0}^{n}W_{\beta,r}\left(
n,k\right)  \frac{\left(  -1\right)  ^{k}k!}{\beta^{n-k}\left(  k+1\right)
},
\]
which was also given in \cite{Mihioubi2} with a different proof. Moreover, for
$\beta=1$ and $r=0$ the above equation reduce to the left part of (\ref{14}).

Now, we give a new explicit formula for Carlitz's degenerate Bernoulli
polynomials in the following theorem.

\begin{theorem}
For every $s\geq0,$
\begin{equation}
\beta_{n+1}\left(  \alpha,r\right)  -\beta_{n+1}\left(  \alpha,r-s\right)
=\left(  n+1\right)  \sum_{k=0}^{n}S_{2}\left(  n,k,r\mid\alpha\right)
\frac{\left(  -1\right)  ^{k}\left\langle s\right\rangle _{k+1}}{k+1}.
\label{29}%
\end{equation}

\end{theorem}

\begin{proof}
If we integrate both sides of (\ref{8}) for $x$ from $-1$ to $0,$ we have%
\begin{equation}
\sum_{n=0}^{\infty}\frac{t^{n}}{n!}%
{\displaystyle\int\limits_{-1}^{0}}
w_{n}^{\left(  s+1\right)  }\left(  x;\alpha,\beta,r\right)  dx=\frac{1}%
{st}\left[  \frac{t\left(  1+\alpha t\right)  ^{r/\alpha}}{\left(  1+\alpha
t\right)  ^{\beta/\alpha}-1}-\frac{t\left(  1+\alpha t\right)  ^{\left(
r-\beta s\right)  /\alpha}}{\left(  1+\alpha t\right)  ^{\beta/\alpha}%
-1}\right]  . \label{13}%
\end{equation}
For $\beta=1$, the above equation becomes%
\[
\sum_{n=0}^{\infty}\frac{t^{n}}{n!}%
{\displaystyle\int\limits_{-1}^{0}}
w_{n}^{\left(  s+1\right)  }\left(  x;\alpha,1,r\right)  dx=\frac{1}{s}\left[
\sum_{n=0}^{\infty}\frac{\beta_{n+1}\left(  \alpha,r\right)  -\beta
_{n+1}\left(  \alpha,r-s\right)  }{n+1}\right]  \frac{t^{n}}{n!}.
\]
Equating the coefficients of $\frac{t^{n}}{n!}$ in the above equation, we
obtain%
\[%
{\displaystyle\int\limits_{-1}^{0}}
w_{n}^{\left(  s+1\right)  }\left(  x;\alpha,1,r\right)  dx=\frac{\beta
_{n+1}\left(  \alpha,r\right)  -\beta_{n+1}\left(  \alpha,r-s\right)
}{s\left(  n+1\right)  }.
\]
Finally, using (\ref{3}) yields the desired equation.
\end{proof}

We note that for $s=1,$ (\ref{29}) becomes
\[
\beta_{n+1}\left(  \alpha,r\right)  -\beta_{n+1}\left(  \alpha,r-1\right)
=\left(  n+1\right)  \left(  r-1\mid\alpha\right)  _{n},
\]
which is extension of well-known identity for Bernoulli polynomials%
\[
B_{n+1}\left(  r\right)  -B_{n+1}\left(  r-1\right)  =\left(  n+1\right)
\left(  r-1\right)  ^{n}.
\]

Taking $s=r$ in (\ref{29}), we have
\[
\beta_{n+1}\left(  \alpha,r\right)  =\beta_{n+1}\left(  \alpha\right)
+\left(  n+1\right)  \sum_{k=0}^{n}S_{2}\left(  n,k,r\mid\alpha\right)
\frac{\left(  -1\right)  ^{k}\left\langle r\right\rangle _{k+1}}{k+1}.
\]
On the other hand, taking $s=r$ in (\ref{29}) and using the relation \cite[Eq.
(5.4)]{Carlitz}%
\[
\sum_{j=0}^{r-1}\left(  j\mid\alpha\right)  _{m}=\frac{1}{m+1}\left[
\beta_{m+1}\left(  \alpha,r\right)  -\beta_{m+1}\left(  \alpha\right)
\right]  ,
\]
we derive the following corollary.

\begin{corollary}
The sums of generalized falling factorials can be expressed as%
\[
\sum_{j=0}^{r-1}\left(  j\mid\alpha\right)  _{n}=\sum_{k=0}^{n}S_{2}\left(
n,k,r\mid\alpha\right)  \frac{\left(  -1\right)  ^{k}\left\langle
r\right\rangle _{k+1}}{k+1},
\]
where $r$ is any integer $>0.$
\end{corollary}

Setting $s=\alpha$ in (\ref{29}) and using the identity \cite[Eq.
(5.10)]{Carlitz}%
\[
\beta_{n}\left(  \alpha,z+\alpha\right)  =\beta_{n}\left(  \alpha,z\right)
+\alpha n\beta_{n-1}\left(  \alpha,z\right)  ,\text{ \ }\alpha>0,
\]
gives the following corollary.

\begin{corollary}
For every $\alpha\geq0,$%
\begin{equation}
\beta_{n}\left(  \alpha,r-\alpha\right)  =\sum_{k=0}^{n}S_{2}\left(
n,k,r\mid\alpha\right)  \frac{\left(  -1\right)  ^{k}\left\langle
\alpha+1\right\rangle _{k}}{k+1}. \label{31}%
\end{equation}

\end{corollary}

The identity given in \cite{Mihioubi2} can also be derived form (\ref{31}) for
$\alpha\rightarrow0.$ Moreover, when $r=0,$ (\ref{31}) becomes
\begin{equation}
\beta_{n}\left(  \alpha,-\alpha\right)  =\sum_{k=0}^{n}S_{2}\left(
n,k\mid\alpha\right)  \frac{\left(  -1\right)  ^{k}\left\langle \alpha
+1\right\rangle _{k}}{k+1}. \label{32}%
\end{equation}
Taking $r=\alpha$ in (\ref{31}) we have
\begin{equation}
\beta_{n}\left(  \alpha\right)  =\sum_{k=0}^{n}S_{2}\left(  n,k,\alpha
\mid\alpha\right)  \frac{\left(  -1\right)  ^{k}\left\langle \alpha
+1\right\rangle _{k}}{k+1}. \label{33}%
\end{equation}

Let us return to (\ref{13}) again. For $\alpha\rightarrow0,$ (\ref{13}) can be
written as%
\begin{align*}
\underset{\alpha\rightarrow0}{\lim}\sum_{n=0}^{\infty}\frac{t^{n}}{n!}%
{\displaystyle\int\limits_{-1}^{0}}
w_{n}^{\left(  s+1\right)  }\left(  x;\alpha,\beta,r\right)  dx  &  =\frac
{1}{st}\left[  \frac{te^{rt}}{e^{\beta t}-1}-\frac{te^{\left(  r-\beta
s\right)  t}}{e^{\beta t}-1}\right] \\
&  =\frac{1}{s}\left[  \sum_{n=0}^{\infty}\frac{B_{n+1}\left(  \frac{r}{\beta
}\right)  -B_{n+1}\left(  \frac{r}{\beta}-s\right)  }{n+1}\beta^{n+1}\right]
\frac{t^{n}}{n!}.
\end{align*}
Since $S_{2}\left(  n,k;0,\beta,r\right)  =W_{\beta,r}\left(  n,k\right)  ,$
we derive the values of Bernoulli polynomials at rational numbers in the
following theorem.

\begin{theorem}
\label{teo1}For every $s\geq0$ and $\beta\neq0,$ we have%
\begin{equation}
B_{n+1}\left(  \frac{r}{\beta}\right)  -B_{n+1}\left(  \frac{r}{\beta
}-s\right)  =\left(  n+1\right)  \sum_{k=0}^{n}W_{\beta,r}\left(  n,k\right)
\frac{\left(  -1\right)  ^{k}\left\langle s\right\rangle _{k+1}}{\beta
^{n+1-k}\left(  k+1\right)  }. \label{37}%
\end{equation}

\end{theorem}

We note that for $\beta=1$ and $r=0$ in (\ref{37}) we have
\[
B_{n+1}\left(  -s\right)  =B_{n+1}+\left(  n+1\right)  \sum_{k=0}^{n}%
\genfrac{\{}{\}}{0pt}{}{n}{k}%
\frac{\left(  -1\right)  ^{k+1}\left\langle s\right\rangle _{k+1}}{k+1}.
\]
Replacing $-s$ with $x$ in the above equation and using (\ref{36}) gives the
well-known relation between Bernoulli polynomials and falling factorial as
\cite[Eq. (15.39)]{Gould}\textbf{ }%
\[
B_{n+1}\left(  x\right)  =B_{n+1}+\sum_{k=0}^{n}\frac{\left(  n+1\right)
}{k+1}%
\genfrac{\{}{\}}{0pt}{}{n}{k}%
\left(  x\right)  _{k+1}.
\]

For $\beta=1$ in (\ref{37}) we have%
\[
B_{n+1}\left(  r\right)  -B_{n+1}\left(  r-s\right)  =\sum_{k=0}^{n}%
\frac{\left(  -1\right)  ^{k}\left(  n+1\right)  }{k+1}%
\genfrac{\{}{\}}{0pt}{}{n+r}{k+r}%
_{r}\left\langle s\right\rangle _{k+1}.
\]
Thus, taking $s=r$ in the above equation gives a new relation between
Bernoulli polynomials and rising factorial in the following corollary.

\begin{corollary}
For every $r\geq0,$ we have
\[
B_{n+1}\left(  r\right)  =B_{n+1}+\sum_{k=0}^{n}\frac{\left(  -1\right)
^{k}\left(  n+1\right)  }{k+1}%
\genfrac{\{}{\}}{0pt}{}{n+r}{k+r}%
_{r}\left\langle r\right\rangle _{k+1}.
\]

\end{corollary}

For the last consequences of Theorem \ref{teo1}, we deal with the Howard's
identity \cite[Eq. (4.3)]{Howard2}%
\[
\sum_{j=0}^{m-1}\left(  r+\beta j\right)  ^{n}=\frac{\beta^{n}}{n+1}\left[
B_{n+1}\left(  m+\frac{r}{\beta}\right)  -B_{n+1}\left(  \frac{r}{\beta
}\right)  \right]  ,
\]
Replacing $-s$ with $m$ in (\ref{37}) and using (\ref{36}) gives that the sums
of powers of integers can be evaluated as in the following corollary.

\begin{corollary}
Let $n$ and $m$ be non-negative integers with $m>0$ and $\beta\neq0.$ Then, we
have%
\[
\sum_{j=0}^{m-1}\left(  r+\beta j\right)  ^{n}=\sum_{k=0}^{n}\frac{\beta
^{k-1}}{k+1}W_{\beta,r}\left(  n,k\right)  \left(  m\right)  _{k+1}.
\]

\end{corollary}

\section{Some examples of series and integrals evaluation}

In this section, several examples for the evaluation of some series and
integrals are given.

For the first example, let us take
\[
f\left(  x\right)  =\cosh\left(  x/\beta\right)  =\sum_{k=0}^{\infty}%
\frac{x^{2k}}{\beta^{2k}\left(  2k\right)  !},
\]
in (\ref{1}). Then, we have%
\begin{align*}
\left(  \beta x^{1-\alpha/\beta}D\right)  ^{n}\left[  x^{r/\beta}\left(
e^{x/\beta}+e^{-x/\beta}\right)  \right]   &  =2\sum_{k=0}^{\infty}\frac
{1}{\beta^{2k}\left(  2k\right)  !}\left(  \beta x^{1-\alpha/\beta}D\right)
^{n}\left[  x^{\left(  2k\beta+r\right)  /\beta}\right] \\
&  =2x^{\left(  r-n\alpha\right)  /\beta}\sum_{k=0}^{\infty}\frac{\left(
2k\beta+r\mid\alpha\right)  _{n}}{\beta^{2k}\left(  2k\right)  !}x^{2k}.
\end{align*}
From (\ref{15}), the left hand side can be written as
\begin{align*}
&  \left(  \beta x^{1-\alpha/\beta}D\right)  ^{n}\left[  x^{r/\beta}\left(
e^{x/\beta}+e^{-x/\beta}\right)  \right] \\
&  ^{\qquad\qquad\qquad\qquad\qquad\qquad\qquad}=\left(  \beta x^{1-\alpha
/\beta}D\right)  ^{n}\left[  x^{r/\beta}e^{x/\beta}\right]  +\left(  \beta
x^{1-\alpha/\beta}D\right)  ^{n}\left[  x^{r/\beta}e^{-x/\beta}\right] \\
&  ^{\qquad\qquad\qquad\qquad\qquad\qquad\qquad}=x^{\left(  r-n\alpha\right)
/\beta}e^{x/\beta}S_{n}\left(  x\right)  +x^{\left(  r-n\alpha\right)  /\beta
}e^{-x/\beta}S_{n}\left(  -x\right)  .
\end{align*}
Thus, we have
\[
2\sum_{k=0}^{\infty}\frac{\left(  2k\beta+r\mid\alpha\right)  _{n}}{\beta
^{2k}\left(  2k\right)  !}x^{2k}=e^{x/\beta}S_{n}\left(  x\right)
+e^{-x/\beta}S_{n}\left(  -x\right)  .
\]
Moreover, setting $x=2\pi i\beta$ in the above equation gives%
\begin{equation}
\sum_{k=0}^{\infty}\left(  2k\beta+r\mid\alpha\right)  _{n}\frac{\left(
-1\right)  ^{k}\left(  2\pi\right)  ^{2k}}{\left(  2k\right)  !}=\sum
_{j=1}^{\left\lfloor n/2\right\rfloor }S\left(  n,2j;\alpha,\beta,r\right)
\left(  -1\right)  ^{j}\left(  2\pi\beta\right)  ^{2j}. \label{17}%
\end{equation}

Similarly, taking $f\left(  x\right)  =\sinh x$ in (\ref{1}) gives%
\[
2\sum_{k=0}^{\infty}\frac{\left(  \left(  2k+1\right)  \beta+r\mid
\alpha\right)  _{n}}{\beta^{2k+1}\left(  2k+1\right)  !}x^{2k+1}=e^{x/\beta
}S_{n}\left(  x\right)  -e^{-x/\beta}S_{n}\left(  -x\right)
\]
and%
\begin{equation}
\sum_{k=0}^{\infty}\left(  \left(  2k+1\right)  \beta+r\mid\alpha\right)
_{n}\frac{\left(  -1\right)  ^{k}\left(  2\pi\right)  ^{2k}}{\left(
2k+1\right)  !}=\beta\sum_{j=1}^{\left\lfloor n/2\right\rfloor }S\left(
n,2j+1;\alpha,\beta,r\right)  \left(  -1\right)  ^{j}\left(  2\pi\beta\right)
^{2j}. \label{18}%
\end{equation}

Note that (\ref{17}) and (\ref{18}) are the generalization of the identities
given in \cite[Page 403]{MezoandRamirez}.

Now, if we apply (\ref{1}) to the both sides of the function%
\[
\frac{1}{\left(  1-x\right)  ^{s+1}}=\sum_{k=0}^{\infty}\binom{s+k}{k}x^{k}%
\]
and use (\ref{4}), we obtain (for $n,s=0,1,2,\ldots$)%
\begin{equation}
\sum_{k=0}^{\infty}\binom{s+k}{k}\left(  r+k\beta\mid\alpha\right)  _{n}%
x^{k}=\frac{1}{\left(  1-x\right)  ^{s+1}}w_{n}^{\left(  s+1\right)  }\left(
\frac{x}{1-x};\alpha,\beta,r\right)  , \label{5}%
\end{equation}
which is the generalization of \cite{Kargin}%
\begin{equation}
\sum_{k=0}^{\infty}\left(  r+k\beta\mid\alpha\right)  _{n}x^{k}=\frac
{1}{\left(  1-x\right)  }w_{n}\left(  \frac{x}{1-x};\alpha,\beta,r\right)  .
\label{21}%
\end{equation}

For the next example, for every $s>0,$ if we take
\[
f\left(  x\right)  =\left(  1+x\right)  ^{s}=\sum_{k=0}^{\infty}\binom{s}%
{k}x^{k}%
\]
in (\ref{1}) and use (\ref{38}), we derive%
\[
\sum_{k=0}^{\infty}\binom{s}{k}\left(  r+k\beta\mid\alpha\right)  _{n}%
x^{k}=\left(  1+x\right)  ^{s}w_{n}^{\left(  -s\right)  }\left(  \frac
{-x}{1+x};\alpha,\beta,r\right)  ,
\]
which is the generalization of the identity in \cite[Eq. 9]{BandDil}.

For the last example of the evaluation of the series in closed forms, we deal
with the digamma function $\psi\left(  x\right)  $. The digamma function
$\psi\left(  x\right)  $ can be given by the Taylor series at $x=1,$ \cite[Eq.
6.3.14]{Abromowitz}
\begin{equation}
\psi\left(  x+1\right)  =-\gamma+\sum_{k=1}^{\infty}\zeta\left(  k+1\right)
\left(  -1\right)  ^{k+1}x^{k},\text{ \ \ }\left\vert x\right\vert <1,
\label{23}%
\end{equation}
where $\zeta\left(  s\right)  $ is the Riemann zeta function and $\gamma$ is
the Euler's constant. Taking (\ref{23}) in (\ref{1}) gives%
\begin{equation}
\left(  \beta x^{1-\alpha/\beta}D\right)  ^{n}\left[  x^{r/\beta}\psi\left(
x+1\right)  \right]  =x^{\left(  r-n\alpha\right)  /\beta}\sum_{k=0}%
^{n}S\left(  n,k;\alpha,\beta,r\right)  \beta^{k}x^{k}\psi^{\left(  k\right)
}\left(  x+1\right)  , \label{24}%
\end{equation}
where $\psi^{\left(  m\right)  }\left(  x\right)  $ is the polygamma function
which can be written more compactly in terms of the Hurwitz zeta function
$\zeta\left(  s,x\right)  $ as\textbf{ }\cite[Eq. 6.4.10]{Abromowitz}%
\begin{equation}
\psi^{\left(  m\right)  }\left(  x\right)  =\left(  -1\right)  ^{m+1}%
m!\zeta\left(  m+1,x\right)  . \label{25}%
\end{equation}
Here $m>0,$ and $x$ is any complex number not equal to a negative integer.
Using (\ref{23}) in the left hand side and (\ref{25}) in the right hand side
of (\ref{24}), we obtain the following theorem.

\begin{theorem}
\label{teo2}For $\left\vert x\right\vert <1,$%
\begin{align}
&  \sum_{k=1}^{\infty}\zeta\left(  k+1\right)  \left(  r+k\beta\mid
\alpha\right)  _{n}x^{k}\label{26}\\
&  \quad=-\left(  r\mid\alpha\right)  _{n}\left(  \psi\left(  1-x\right)
+\gamma\right)  +\sum_{k=1}^{n}S\left(  n,k;\alpha,\beta,r\right)
k!\zeta\left(  k+1,1-x\right)  \left(  \beta x\right)  ^{k}.\nonumber
\end{align}

\end{theorem}

As a consequences of Theorem \ref{teo2}, the following sums are obtained:

For $\left(  \alpha,\beta,r\right)  =\left(  0,\beta,r\right)  ,$%
\[
\sum_{k=1}^{\infty}\zeta\left(  k+1\right)  \left(  r+k\beta\right)  ^{n}%
x^{k}=-r^{n}\left(  \psi\left(  1-x\right)  +\gamma\right)  +\sum_{k=1}%
^{n}W_{\beta,r}\left(  n,k\right)  k!\zeta\left(  k+1,1-x\right)  \left(
\beta x\right)  ^{k}.
\]

For $\left(  \alpha,\beta,r\right)  =\left(  0,1,r\right)  ,$%
\[
\sum_{k=1}^{\infty}\zeta\left(  k+1\right)  \left(  r+k\right)  ^{n}%
x^{k}=-r^{n}\left(  \psi\left(  1-x\right)  +\gamma\right)  +\sum_{k=1}^{n}%
\genfrac{\{}{\}}{0pt}{}{n+r}{k+r}%
_{r}k!\zeta\left(  k+1,1-x\right)  x^{k}.
\]

For $\left(  \alpha,\beta,r\right)  =\left(  0,1,1\right)  ,$ using the
relation $%
\genfrac{\{}{\}}{0pt}{}{n}{k}%
_{0}=%
\genfrac{\{}{\}}{0pt}{}{n}{k}%
_{1}=%
\genfrac{\{}{\}}{0pt}{}{n}{k}%
,$%

\begin{equation}
\sum_{k=1}^{\infty}\zeta\left(  k+1\right)  \left(  k+1\right)  ^{n}%
x^{k}=-\left(  \psi\left(  1-x\right)  +\gamma\right)  +\sum_{k=1}^{n}%
\genfrac{\{}{\}}{0pt}{}{n+1}{k+1}%
k!\zeta\left(  k+1,1-x\right)  x^{k}. \label{30}%
\end{equation}

Finally, setting $x=1/2$ in the above equation, using $\psi\left(  1/2\right)
=-\gamma-2\log2$ and $\zeta\left(  s,1/2\right)  =\left(  2^{s}-1\right)
\zeta\left(  s\right)  $ yields%
\[
\sum_{k=2}^{\infty}\frac{\zeta\left(  k\right)  k^{n}}{2^{k}}=\log2+\sum
_{k=1}^{n}%
\genfrac{\{}{\}}{0pt}{}{n+1}{k+1}%
k!\left(  1-2^{-k-1}\right)  \zeta\left(  k+1\right)  .
\]
Therefore, some special cases of the above equation involving infinite sums of
Riemann zeta function can be listed as%
\begin{align*}
\sum_{k=2}^{\infty}\frac{\zeta\left(  k\right)  }{2^{k}}  &  =\log2,\\
\sum_{k=2}^{\infty}\frac{\zeta\left(  k\right)  k}{2^{k}}  &  =\log2+\frac
{3}{4}\zeta\left(  2\right)  ,\\
\sum_{k=2}^{\infty}\frac{\zeta\left(  k\right)  k^{2}}{2^{k}}  &  =\log
2+\frac{9}{4}\zeta\left(  2\right)  +\frac{14}{8}\zeta\left(  3\right)  .
\end{align*}

Note that similar result of (\ref{30}) can be found in \cite[Proposition
20]{BandDil}.

Now, we want to prove the equation in (\ref{7}) by using (\ref{1}). Applying
(\ref{1}) to the both sides of the integral%
\[
\frac{1}{\left(  1-x\right)  ^{s}}=\frac{1}{\Gamma\left(  s\right)  }%
{\displaystyle\int\limits_{0}^{\infty}}
t^{s-1}e^{-\left(  1-x\right)  t}dt
\]
and using (\ref{4}) we obtain%
\[
\frac{x^{\left(  r-n\alpha\right)  /\beta}}{\left(  1-x\right)  ^{s+1}}%
w_{n}^{\left(  s+1\right)  }\left(  \frac{x}{1-x};\alpha,\beta,r\right)
=\frac{1}{\Gamma\left(  s\right)  }%
{\displaystyle\int\limits_{0}^{\infty}}
t^{s-1}e^{-t}\left(  \beta x^{1-\alpha/\beta}D\right)  ^{n}\left[  x^{r/\beta
}e^{xt}\right]  dt.
\]
At the same time, we have%
\begin{align*}
\left(  \beta x^{1-\alpha/\beta}D\right)  ^{n}\left[  x^{r/\beta}%
e^{xt}\right]   &  =\left(  \beta x^{1-\alpha/\beta}D\right)  ^{n}\sum
_{k=0}^{\infty}\frac{t^{k}}{k!}x^{k+r/\beta}\\
&  =\sum_{k=0}^{\infty}\frac{t^{k}}{k!}\left(  \beta x^{1-\alpha/\beta
}D\right)  ^{n}x^{k+r/\beta}\\
&  =x^{\left(  r-n\alpha\right)  /\beta}\sum_{k=0}^{\infty}\frac{t^{k}}%
{k!}\left(  k\beta+r\mid\alpha\right)  _{n}x^{k}\\
&  =x^{\left(  r-n\alpha\right)  /\beta}e^{xt}S_{n}\left(  x\beta t\right)  .
\end{align*}
Thus, we have a different representation of (\ref{7}) as%
\[
\frac{1}{\left(  1-x\right)  ^{s+1}}w_{n}^{\left(  s+1\right)  }\left(
\frac{x}{1-x};\alpha,\beta,r\right)  =\frac{1}{\Gamma\left(  s\right)  }%
{\displaystyle\int\limits_{0}^{\infty}}
t^{s-1}S_{n}\left(  x\beta t\right)  e^{-\left(  1-x\right)  t}dt.
\]

Now, we want to add some examples for the evaluation of integrals. Before
giving the examples we need to mention that in the rest of this section we use
the well-known estimate for the gamma function:%
\[
\left\vert \Gamma\left(  x+iy\right)  \right\vert \sim\sqrt{2\pi}\left\vert
y\right\vert ^{x-\frac{1}{2}}e^{-x-\frac{\pi}{2}\left\vert y\right\vert },
\]
($\left\vert y\right\vert \rightarrow\infty$) for any fixed real $x.$ This
explains the behavior of the gamma function on vertical lines $\left\{
t=a+iz:-\infty<z<\infty,\text{ }0<a<1\right\}  $.

For the first example, let us start from the Mellin integral representation
\cite[Formula 5.37]{Oberhettinger}%

\[
\frac{1}{\left(  1+x\right)  ^{s+1}}=\frac{1}{2\pi i\Gamma\left(  s+1\right)
}%
{\displaystyle\int\limits_{a-i\infty}^{a+i\infty}}
x^{-t}\Gamma\left(  t\right)  \Gamma\left(  s+1-t\right)  dt,
\]
where $s\geq0,$ $0<x<1.$ Apply (\ref{1}) to the both sides of the above
integral to obtain
\[
\left(  \beta x^{1-\alpha/\beta}D\right)  ^{n}\left[  \frac{x^{r/\beta}%
}{\left(  1+x\right)  ^{s+1}}\right]  =\frac{x^{\left(  r-n\alpha\right)
/\beta}}{2\pi i\Gamma\left(  s+1\right)  }%
{\displaystyle\int\limits_{a-i\infty}^{a+i\infty}}
\left(  r-\beta t\mid\alpha\right)  _{n}x^{-t}\Gamma\left(  t\right)
\Gamma\left(  s+1-t\right)  dt.
\]
From the left hand side of the above equation, we derive%
\begin{align*}
\left(  \beta x^{1-\alpha/\beta}D\right)  ^{n}\left[  \frac{x^{r/\beta}%
}{\left(  1+x\right)  ^{s+1}}\right]   &  =\sum_{k=0}^{\infty}\binom{s+k}%
{k}\left(  -1\right)  ^{k}\left(  \beta x^{1-\alpha/\beta}D\right)
^{n}x^{\left(  k\beta+r\right)  /\beta}\\
&  =x^{\left(  r-n\alpha\right)  /\beta}\sum_{k=0}^{\infty}\binom{s+k}%
{k}\left(  r+k\beta\mid\alpha\right)  _{n}\left(  -x\right)  ^{k}\\
&  =\frac{x^{\left(  r-n\alpha\right)  /\beta}}{\left(  1+x\right)  ^{s+1}%
}w_{n}^{\left(  s+1\right)  }\left(  \frac{-x}{1+x};\alpha,\beta,r\right)  .
\end{align*}
Therefore, we have
\begin{equation}
\frac{1}{\left(  1+x\right)  ^{s+1}}w_{n}^{\left(  s+1\right)  }\left(
\frac{-x}{1+x};\alpha,\beta,r\right)  =\frac{1}{2\pi i\Gamma\left(
s+1\right)  }%
{\displaystyle\int\limits_{a-i\infty}^{a+i\infty}}
\left(  r-\beta t\mid\alpha\right)  _{n}x^{-t}\Gamma\left(  t\right)
\Gamma\left(  s+1-t\right)  dt, \label{12}%
\end{equation}

Secondly, replace $x$ by $\left(  r+k\beta\right)  x$, multiply both sides by
$\left(  r+k\beta\mid\alpha\right)  _{n}$ and sum for $k=0,1,\ldots,$ in the
integral:%
\[
e^{-x}=\frac{1}{2\pi i}%
{\displaystyle\int\limits_{a-i\infty}^{a+i\infty}}
x^{-t}\Gamma\left(  t\right)  dt.
\]
Then, we have%
\[
e^{-rx}\sum_{k=0}^{\infty}\left(  r+k\beta\mid\alpha\right)  _{n}\left(
e^{-\beta x}\right)  ^{k}=\frac{1}{2\pi i}%
{\displaystyle\int\limits_{a-i\infty}^{a+i\infty}}
x^{-t}\sum_{k=0}^{\infty}\frac{\left(  k\beta+r\mid\alpha\right)  _{n}%
}{\left(  k\beta+r\right)  ^{t}}\Gamma\left(  t\right)  dt.
\]
From (\ref{21}), the above integral becomes%
\[
\frac{e^{-rx}}{1-e^{-\beta x}}w_{n}\left(  \frac{e^{-\beta x}}{1-e^{-\beta x}%
};\alpha,\beta,r\right)  =\frac{1}{2\pi i}%
{\displaystyle\int\limits_{a-i\infty}^{a+i\infty}}
x^{-t}\sum_{k=0}^{\infty}\frac{\left(  k\beta+r\mid\alpha\right)  _{n}%
}{\left(  k\beta+r\right)  ^{t}}\Gamma\left(  t\right)  dt.
\]
The interesting part of the above integral, given in the following
proposition, is appeared when $\alpha\rightarrow0.$

\begin{proposition}
For all $x>0,$ $\operatorname{Re}\left(  \beta\right)  >0,$ $\operatorname{Re}%
\left(  r\right)  >0$ $n=0,1,2,\ldots$ and $a>n+1$%
\begin{equation}
\frac{e^{-rx}}{1-e^{-\beta x}}w_{n}\left(  \frac{e^{-\beta x}}{1-e^{-\beta x}%
};0,\beta,r\right)  =\frac{\beta^{n}}{2\pi i}%
{\displaystyle\int\limits_{a-i\infty}^{a+i\infty}}
\left(  \beta x\right)  ^{-t}\zeta\left(  t-n,\frac{r}{\beta}\right)
\Gamma\left(  t\right)  dt. \label{22}%
\end{equation}

\end{proposition}

We note that for $\alpha\rightarrow0,$ generalized geometric polynomials
become
\[
w_{n}\left(  x;0,\beta,r\right)  =\sum_{k=0}^{n}W_{\beta,r}\left(  n,k\right)
\beta^{k}k!x^{k}.
\]

Moreover, for $\beta=1$ and $n=0$ in (\ref{22}), we have the well-known
inverse Mellin transformation of Hurwitz zeta function \cite[Theorem
12.2]{Apostol}. Moreover, (\ref{12}) and (\ref{22}) are the generalization of
identities in \cite[Proposition 16]{BandDil}.

\end{document}